\documentclass{amsproc}

\usepackage{amsmath}
\usepackage{amssymb}
\usepackage{rotating}
\usepackage{graphicx}
\usepackage[tableposition=below]{caption}
\usepackage{bm}

\DeclareGraphicsExtensions{.png,.pdf,.eps}
\usepackage[all,2cell,cmtip]{xy}
\usepackage{tikz}
\usepackage{color}
\usepackage{longtable}
\usepackage{soul}
\newtheorem{theorem}{Theorem}[section]

\newtheorem{corollary}[theorem]{Corollary}
\newtheorem{proposition}[theorem]{Proposition}

\theoremstyle{definition}

\newtheorem{definition}[theorem]{Definition}

\newtheorem{example}[theorem]{Example}
\newtheorem{question}[theorem]{Question}

\def\G{{\mathbb G}}

\def\P{{\mathbb P}}

\def\R{{\mathbb R}}
\def\Z{{\mathbb Z}}

\def\cE{{\mathcal E}}

\def\cL{{\mathcal L}}
\def\cM{{\mathcal M}}
\def\cN{{\mathcal{N}}}
\def\cO{{\mathcal{O}}}

\def\G{{\mathbb{G}}}

\def\cOperatorname#1{\mathop{\rm #1}\nolimits}

\def\deg{\cOperatorname{deg}}

\def\det{\cOperatorname{det}}

\def\NE{{\cOperatorname{NE}}}
\def\ME{{\cOperatorname{ME}}}

\newcommand{\cME}[1]{\cOverline{\ME}}

\makeindex

\begin{document}
%\pagewiselinenumbers

\title{Fano $4$-folds with nef tangent bundle in positive characteristic}

\author{Yuta Takahashi}
\author{Kiwamu Watanabe}
\date{\today}
\address{Department of Mathematics, Faculty of Science and Engineering, Chuo University.
1-13-27 Kasuga, Bunkyo-ku, Tokyo 112-8551, Japan}
\email{watanabe@math.chuo-u.ac.jp}
\email{yuta0630takahashi0302@gmail.com}
\thanks{The author is partially supported by JSPS KAKENHI Grant Number 21K03170 and the Sumitomo Foundation Grant Number 190170.}


\begin{abstract} In characteristic $0$, the Campana-Peternell conjecture claims that the only smooth Fano variety with nef tangent bundle should be homogeneous. In this paper, we study the positive characteristic version of the Campana-Peternell conjecture. In particular, we give an affirmative answer for Fano $4$-folds with nef tangent bundle and Picard number greater than one.  
\end{abstract}

\maketitle

\section{Introduction} How can one compare two given smooth projective varieties? Since any smooth variety $X$ has the tangent bundle $T_X$, we often use the tangent bundle $T_X$ to compare smooth varieties. In particular, the positivity of the tangent bundle imposes strong restrictions on the geometry of varieties. For instance, in the celebrated paper \cite{Mori79}, Mori solved the famous Hartshorne conjecture. The Hartshorne conjecture states that a smooth projective variety $X$ defined over an algebraically closed field is the projective space if $T_X$ is ample. As a generalization of the Hartshorne conjecture, Campana and Peternell studied complex smooth projective varieties with nef tangent bundle \cite{CP91}. In this direction, Demailly-Peternell-Schneider \cite{DPS94} proved that any complex smooth projective variety with nef tangent bundle is, up to an \'etale cover,  a Fano fiber space over an Abelian variety. As a consequence, the study of complex smooth projective varieties with nef tangent bundle can be reduced to that of Fano varieties. Moreover Campana and Peternell \cite{CP91} conjectured that any complex smooth Fano variety with nef tangent bundle is homogeneous. This conjecture holds for varieties of dimension at most five, but in general this is widely open. We refer the reader to \cite{MOSWW}.  

In \cite{KW21}, the second author and Kanemitsu proved an analogue of the theorem by Demailly-Peternell-Schneider in positive characteristic; thus the next step is to study smooth Fano varieties with nef tangent bundle in positive characteristic. When the dimension is at most three, this problem was studied by the second author \cite{Wat17}. In the present paper, we give a classification of Fano $4$-folds with nef tangent bundle and Picard number greater than one:

\begin{theorem}\label{them:main} Let $X$ be a smooth Fano $4$-fold defined over an algebraically closed field. If the tangent bundle $T_X$ is nef and the Picard number of $X$ is greater than one, then $X$ is isomorphic to one of the following:
\begin{enumerate}
\item $\P^3\times \P^1$;
\item $Q^3\times \P^1$;
\item $\P^2\times \P^2$; 
\item $\P^2\times \P^1\times \P^1$;
\item $\P(T_{\P^2})\times \P^1$, where $T_{\P^2}$ is the tangent bundle of $\P^2$;
%\item $\P(p^{\ast}T_{\P^2})$ with $p: \P(T_{\P^2})\to \P^2$ is the projection;
\item $\P^1\times \P^1\times\P^1\times \P^1$; 
\item $\P(\cN)$ with a null-correlation bundle $\cN$ on $\P^3$ (see Definition~\ref{def:null}).
\end{enumerate}
In particular, $X$ is a homogeneous variety with reduced stabilizer. 
\end{theorem}

In characteristic $0$, this was proved by Camapana and Peternell \cite{CP93}. However there are some difficulties to study 
this kind of classification problem in positive characteristic. For instance, the proof of \cite{CP93} heavily depends on the Kodaira vanishing theorem and Hodge theory, which unfortunately fail in positive characteristic. We shall give a characteristic-free proof of \cite{CP93}. 

The contents of this paper are organized as follows. In Section 2, we recall the background of our problem. We also review some known properties of Fano varieties with nef tangent bundle, paying special attention to some results in \cite{KW21}. In Section 3, we shall study Fano varieties with nef tangent bundle which admit a projective bundle structure; this study plays a crucial role in the proof of Theorem~\ref{them:main}. In Section 4, we will give a proof of Theorem~\ref{them:main}.

\section{Preliminaries} 

\subsection*{Notations} Let $k$ be an algebraically closed field of characteristic $p\geq 0$. Throughout this paper, we work over $k$ and use standard notations as in \cite{Har, Kb, KM, L1, L2}. For a smooth projective variety $X$, we also use the following notations:
\begin{itemize}
\item We denote by $T_X$  the tangent bundle of $X$.
\item We denote by $A_k(X)=A^{n-k}(X)$ the group of rational equivalence classes of algebraic $k$-cycles on $X$. We denote by  $A(X):=\bigoplus_k A_k(X)$ the Chow ring of $X$.
\item We denote by $N_1(X)$ the group of numerical equivalence classes of algebraic $1$-cycles with real coefficients on $X$. The dimension $\dim_{\R} N_1(X)$ as an $\R$-vector space is called {\it the Picard number of $X$} and we denote it by $\rho_X$.
\item We say that a smooth projective variety $X$ is {\it Fano} if $-K_X$ is ample. For a smooth Fano variety $X$, the {\it pseudoindex} $\iota_X$ of $X$ is the minimal anticanonical degree of rational curves on $X$.   
\item An {\it $F$-bundle} is a smooth morphism $f:Y \to X$ between smooth projective varieties whose fibers are isomorphic to $F$.
\item An {\it elementary contraction} means a contraction of an extremal ray.
\item For a vector bundle $\cE$ (resp. $\cE_i$) on $X$, we denote the tautological divisor of $\P(\cE)$ (resp. $\P(\cE_i)$) by ${\xi}_{\cE}$ (resp. ${\xi}_{\cE_i}$). When no confusion is likely, we also simply denote the divisor ${\xi}_{\cE}$ (resp. ${\xi}_{\cE_i}$) by $\xi$ (resp. ${\xi}_{i}$).
\item For a rank two vector bundle $\cE$ on $X$, we denote the $i$-th Chern class of $\cE$ by $c_i(\cE)$. When $A^1(X)$ and $ A^2(X)$ are isomorphic to $\Z$, there exist an effective divisor $H$ and an effective $2$-cocycle $L$ on $X$ such that $A^1(X)\cong \Z[H]$ and $A^2(X)\cong \Z[L]$; then we consider $c_1(\cE)$ and $c_2(\cE)$ as integers $c_1$ and $c_2$, that is, $c_1(\cE)=c_1H\in A^1(X)$ and $c_2(\cE)=c_2L\in A^2(X)$. In this setting, we say that $\cE$ is {\it normalized} if $c_1=0$ or $-1$. We also say that $\cE$ is {\it stable} (resp. {\it semistable}) if for every invertible subsheaf $\cL$ of $\cE$, $c_1(\cL)<\dfrac{1}{2} c_1(\cE)$ (resp. $c_1(\cL)\leq \dfrac{1}{2} c_1(\cE)$). 
\item For a vector bundle $\cE$ on $X$, we say that $\cE$ is {\it Fano} if $\P(\cE)$ is a Fano variety.
\item For a vector bundle $\cE$ on $X$, we say that $\cE$ is {\it numerically flat} if $\cE$ and its dual $\cE^{\vee}$ are nef (equivalently $\cE$ and $\det(\cE^{\vee})$ are nef). 
\item For a projective variety $X$, we denote by ${\rm RatCurves}^n(X)$ the {\it family of rational curves} on $X$ (see \cite[II~Definition~2.11]{Kb}).
\item We denote by $\P^n$ and $Q^n$ projective $n$-space and a smooth quadric hypersurface in $\P^{n+1}$ respectively.
\end{itemize}

\subsection{Background of the Problem}

Let $X$ be a smooth projective variety with nef tangent bundle. By the decomposition theorem \cite[Theorem~1.7]{KW21}, $X$ admits a smooth contraction $\varphi: X\to M$ such that 
\begin{itemize}
\item any fiber of $\varphi$ is a  smooth Fano variety with nef tangent bundle;
\item the tangent bundle $T_M$ is numerically flat.
\end{itemize}

This result suggests to study two special cases:

\begin{question}[{\cite[Question 1.8]{KW21}, \cite[Conjecture~11.1]{CP91}, \cite[Question 1]{Wat17}}]\label{ques:str}
Let $X$ be a smooth projective variety with nef tangent bundle.
\begin{enumerate}
 \item If $X$ is a Fano variety, then is $X$ a homogeneous space with reduced stabilizer?
 \item If $T_X$ is numerically flat, then is $X$ an \'etale quotient of an Abelian variety? 
\end{enumerate}
\end{question}

In characteristic zero, for special varieties, including Fano varieties whose dimension is at most five, affirmative answers to the first question are known (see \cite{CP91,CP93,Hwa06,Kan16,Kan17,Kan19, Mok02,MOSW15,SW04,Wat14,Wat15,Wat20}), and an affirmative answer to the second question also follows from the Beauville-Bogomolov decomposition. On the other hand, very little is known in positive characteristic; we refer the reader to \cite{Jos21,Lan15, MS87,Wat17}. Here we only recall the following:

\begin{theorem}[{\cite{CP91, Wat17}}]\label{them:3-fold} 
Let $X$ be a smooth Fano $n$-fold with nef tangent bundle. If $n$ is at most three, then $X$ is one of the following:
\begin{enumerate}
\item $X$ is the $n$-dimensional projective space $\P^n$;
\item $X$ is an $n$-dimensional hyperquadric $Q^n$ $(n=2, 3)$;
\item $X=\P^2 \times \P^1$;
\item $X=\P^1 \times \P^1 \times \P^1$;
\item $X=\P(T_{\P^2})$. 
\end{enumerate}
\end{theorem} 

To give a classification of complex Fano varieties with Picard number greater than one, it is quite common to study extremal contractions, but in positive characteristic, the existence of a contraction of an extremal ray is not known in general. The following result states that there exists a contraction of an extremal ray for Fano varieties with nef tangent bundle:

\begin{theorem}[{a special case of \cite[Theorem~1.5]{KW21}}]\label{them:KW21:propert} Let $X$ be a smooth Fano variety $X$ with nef tangent bundle. Let $R \subset \NE(X)$ be an extremal ray. Then the contraction $f: X\to Y$ of the ray $R$ exists and the following hold:
\begin{enumerate}
\item $f$ is smooth;
\item any fiber $F$ of $f$ is again a smooth Fano variety with nef tangent bundle;
\item $Y$ is also a smooth Fano variety with nef tangent bundle;
\item $\rho_X=\rho_Y+1$ and $\rho_ F=1$.
\end{enumerate} 
\end{theorem}

Let $X$ be a smooth projective variety. We say that $X$ is {\it rationally chain connected} (resp. {\it rationally connected}) if two general points on $X$ can be connected by a connected chain of rational curves (resp. by a single rational curve); it follows from \cite{Cam92}, \cite[Theorem~3.3]{KMM92} that smooth Fano varieties are rationally chain connected (see also \cite[Chapter V.~Theorem 2.13]{Kb}). We say that $X$ is {\it separably rationally connected} if there exists a rational curve $f: \P^1 \to X$ such that $f^{\ast}T_X$ is ample. 
In general, if $X$ is separably rationally connected, then it is rationally connected; by definition, a rationally connected variety is rationally chain connected; moreover these notions coincide in characteristic zero, whereas there exists a rationally connected variety which is not separably rationally connected in characteristic $p>0$ (see for instance \cite[V.~Exercise~5.19]{Kb}). For varieties with nef tangent bundle, these notions coincide:
\begin{theorem}[{\cite[Theorem~1.3,~Theorem~1.6]{KW21}}]\label{them:KW21:RC} For a smooth projective variety $X$ with nef tangent bundle, the following are equivalent to each other:
\begin{enumerate}
\item $X$ is separably rationally connected;
\item $X$ is rationally connected;
\item $X$ is rationally chain connected;
\item $X$ is a Fano variety.
\end{enumerate} 
Moreover, if $X$ satisfies the above equivalent conditions, the Kleiman-Mori cone $\NE(X)$ is simplicial.
\end{theorem}

We also have the following:

\begin{theorem}[{a special case of \cite[Corollary~1.4]{KW21}}]\label{cor:SRC}
For a smooth Fano variety $X$ with nef $T_X$, the following hold:
\begin{enumerate}
 \item $X$ is algebraically simply connected;
 \item $H^1(X, \cO_X) = 0$;
 \item every numerically flat vector bundle on $X$ is trivial.
\end{enumerate}
\end{theorem}

\subsection{Minimal birational sections}
In this subsection, we recall minimal birational sections whose idea appeared in \cite{Wat21, KW21}. Let $X$ be a smooth Fano variety with nef tangent bundle. Assume that $f: X\to Y$ is an extremal contraction and $\dim Y>0$. Since $f$ is a composition of contractions of extremal rays, Theorem~\ref{them:KW21:propert} tells us that $Y$ and any fiber of $f$ are smooth Fano varieties with nef tangent bundle. By Theorem~\ref{them:KW21:RC}, we see that any fiber $F$ of $f$ is separably rationally connected. 
\begin{definition}\label{def:} Under the above notation, let $C\subset Y$ be a rational curve.
We call a rational curve $\tilde{C} \subset X$ a {\it birational section} of $f$ over $C$ if $f|_{\tilde{C}}: \tilde{C}\to C$ is birational. A birational section $\tilde{C} \subset X$ of $f$ over $C$ is {\it minimal} if the anticanonical degree $-K_X\cdot \tilde{C}$ is minimal among birational sections of $f$ over $C$. 
\end{definition}

Let us take a rational curve $\ell \subset Y$ such that $-K_Y\cdot \ell=\iota_Y$ and let $\P^1\to \ell \subset Y$ be the normalization of $\ell$. We consider the fiber product:
\[
  \xymatrix{
    X_{\ell} \ar[r]^i \ar[d]_{f_{\ell}} &  X\ar[d]^{f} \\
    \P^1  \ar[r]&  Y }
\]

By the theorem of de Jong and Starr \cite{DJS03}, $f_{\ell}$ admits a section $\tilde{\ell}$. Let us denote $i(\tilde{\ell})$ by $\tilde{\ell}_X$; then $f|_{\tilde{\ell}_X}:\tilde{\ell}_X \to \ell$ is birational; thus $\tilde{\ell}_X$ is a birational section of $f$ over $\ell$. This yields that there exists a minimal birational section of $f$ over $\ell$. As a consequence, we may find a rational curve $\ell_0 \subset Y$ and a minimal birational section of $f$ over $\ell_0$ satisfying the following:
\begin{itemize}
\item $-K_Y\cdot \ell_0=\iota_Y$;
\item $-K_X\cdot \tilde{\ell_0}_X= \min \left\{  \deg_{(-K_X)}\tilde{\ell}_X \mid -K_Y\cdot \ell=\iota_Y, [\ell] \in {\rm RatCurves}^n(X)  \right\}$.
\end{itemize}
Let us consider a family of rational curves $\cM\subset {\rm RatCurves}^n(X)$ containing $[\tilde{\ell_0}_X]$. By the same argument as in \cite[Proposition~4.14]{Wat21}, we see that $\cM$ is unsplit, that is, $\cM$ is proper as a scheme. Moreover \cite[Proposition~4.4]{KW21} implies that $\R_{\geq 0}[\tilde{\ell_0}_X]$ is an extremal ray and the contraction $g: X\to Z$ of the ray $\R_{\geq 0}[\tilde{\ell_0}_X]$ is a smooth geometric quotient for $\cM$ in the sense of \cite{BCD}. Summing up, we obtain the following:

\begin{proposition}\label{prop:bir:sec:contr} Let $X$ be a smooth Fano variety with nef tangent bundle. Assume that $f: X\to Y$ is an extremal contraction and $\dim Y>0$. Then there exists an unsplit covering family of rational curves $\cM \subset {\rm RatCurves}^n(X)$ such that 
\begin{itemize}
\item for any $[\tilde{\ell}_X]\in \cM$, $f|_{\tilde{\ell}_X}:\tilde{\ell}_X \to f(\tilde{\ell}_X)$ is birational and $-K_Y\cdot f(\tilde{\ell}_X)=\iota_Y$;
\item $\R_{\geq 0}[\tilde{\ell_0}_X]$ is an extremal ray and the contraction $g: X\to Z$ of the ray $\R_{\geq 0}[\tilde{\ell_0}_X]$ is a smooth geometric quotient for $\cM$.
\end{itemize}
\end{proposition}

\begin{definition}[{[cf.~\cite[Definition~1]{MOSW15}]}]\label{def:FT}  Let $X$ be a smooth projective variety with nef tangent bundle. We say that $X$ is an {\it FT-manifold} if every elementary contraction of $X$ is a $\P^1$-bundle.  
\end{definition}

\begin{example}\label{eg:FT} The variety $\P(T_{\P^2})$ is isomorphic to a hyperplane section of a Segre $4$-fold $\P^2\times \P^2 \subset \P^8$. Since $\P(T_{\P^2})$ admit two $\P^1$-bundle structures over $\P^2$ and $\rho_{\P(T_{\P^2})}=2$, it is an FT-manifold. The projective line $\P^1$ is also a basic example of an FT-manifold.
\end{example}

\begin{proposition}\label{prop:dim1} Let $X$ be a smooth Fano variety with nef tangent bundle. Assume that $f: X\to Y$ is an extremal contraction and $\dim Y=1$. Then $X$ is isomorphic to a product of $\P^1$ and a variety $Z$.
\end{proposition}

\begin{proof} We employ the notation as in Proposition~\ref{prop:bir:sec:contr}. Remark that $Y$ is $\P^1$. The contraction $g: X\to Z$ is a $\P^1$-bundle; moreover any fiber of $g$ is a section of $f$; this yields that $X$ is isomorphic to a product of $\P^1$ and $Z$.
\end{proof}

\begin{corollary}\label{cor:FT} Let $M$ be a smooth Fano variety with nef tangent bundle. Assume that $f: M\to X$ is an extremal contraction onto an FT-manifold $X$. Then $M$ is isomorphic to a product of $X$ and a variety $Y$.
\end{corollary}

\begin{proof} This follows from Proposition~\ref{prop:dim1}, Theorem~\ref{them:KW21:propert} and the same argument as in \cite[Proposition~5]{MOSW15}.
\if0
We follow the idea of \cite[Proposition~5]{MOSW15}. For the reader’s convenience, we describe the idea of the proof. 

Let $\sigma$ be the extremal face corresponding to $f$. We denote by $R_1, \ldots, R_n$ all extremal rays which are not contained in the face $\sigma$; we denote by $\overline{\sigma}$ the extremal face generated by $R_1, \ldots, R_n$. Let $g:  M\to Y$ (resp. $\tau_i: X\to X_i$~($i=1, 2, \ldots, n$)) be the contraction of the face $\overline{\sigma}$ (resp. $\langle \sigma, R_i\rangle$). Since the Kleiman-Mori cone $\NE(M)$ is simplicial by Theorem~\ref{them:KW21:RC}, we have $$\NE(X)=f_{\ast}(R_1)+\ldots+f_{\ast}(R_n).$$ 
For the contraction $\pi_i: X\to X_i$ of the extremal ray $f_{\ast}(R_i)$, we denote by $\Gamma_i$ a fiber of $\pi_i$. Since $f^{-1}(\Gamma_i)$ is a smooth Fano variety with nef tangent bundle (see for instance \cite[Proposition~3.6]{Wat14}), Proposition~\ref{prop:dim1} implies that $f^{-1}(\Gamma_i)\cong \Gamma_i\times Z$ for some variety $Z$. 
\fi
\end{proof}

\subsection{Projective bundles}
\begin{definition}[{\cite[Definition~3.2]{CS21}}]\label{def:Brauer} The (cohomological) Brauer group of a scheme $Y$ is ${\rm Br}(Y):=H_{\acute{e}t}^2(Y, {\G}_m)$.
\end{definition}

\begin{proposition}\label{prop:Brauer} Let $f: X\to Y$ be a $\P^n$-bundle. If the Brauer group 
${\rm Br}(Y)$ vanishes, then there exists a vector bundle $\cE$ of rank $n+1$ on $Y$ such that $X\cong \P(\cE)$.
\end{proposition}

\begin{proof} See for instance \cite{MurSE}.
\end{proof}

\begin{corollary}\label{cor:Brauer} Let $f: X\to Y$ be a $\P^n$-bundle. If $Y$ is rational, then there exists a vector bundle $\cE$ of rank $n+1$ on $Y$ such that $X\cong \P(\cE)$.
\end{corollary}

\begin{proof} By \cite[Theorem~5.1.3, Proposition~5.2.2]{CS21}, we see that $${\rm Br}(Y)\cong {\rm Br}(\P^n_k)\cong {\rm Br}(k).$$ Then \cite[Corollary~1.2.4]{CS21} implies that ${\rm Br}(k)$ vanishes; thus our assertion follows from Proposition~\ref{prop:Brauer}.
\end{proof}

\section{Fano bundles over $\P^2$, $\P^3$ and $Q^3$} 

The major difficulty of the proof of Theorem~\ref{them:main} is to study the cases where a smooth Fano $4$-fold with nef tangent bundle admits a $\P^2$-bundle structure over $\P^2$ or a $\P^1$-bundle structure over $\P^3$ and over $Q^3$. In this section, we shall study such cases. 

\subsection{Rank three Fano bundles over $\P^2$} 

\begin{proposition}\label{prop:P2-bundle/P2} Let $X$ be a smooth Fano $4$-fold with nef tangent bundle. Assume that $f_1: X \to \P^2$ is a $\P^2$-bundle. Then $X$ is isomorphic to $\P^2\times \P^2$.
\end{proposition}

\begin{proof} By Theorem~\ref{them:KW21:propert}, we may find another smooth elementary contraction $f_2: X\to X_2$ besides $f_1$. Applying Theorem~\ref{them:3-fold}, Theorem~\ref{them:KW21:propert} and \cite{Sato85}, we see that $f_2: X\to X_2$ is a $\P^1$-bundle over $Q^3$ or a $\P^2$-bundle over $\P^2$. We claim that $f_2: X\to X_2$ is not a $\P^1$-bundle over $Q^3$. To prove this, assume that $f_2: X\to X_2$ is a $\P^1$-bundle over $Q^3$. Then by Corollary~\ref{cor:Brauer}, $f_1$ and $f_2$ are given by the projectivizations of vector bundles. Let us consider the Chow ring of $X$. Since $f_1: X \to \P^2$ is a $\P^2$-bundle over $\P^2$, \cite[Theorem~9.6]{EH16} tells us that the rank of the $A^3(X)$ is three; however, since $f_2: X \to Q^3$ is a $\P^1$-bundle over $Q^3$, \cite[Theorem~9.6]{EH16} tells us that the rank of the $A^3(X)$ is two; this is a contradiction. As a consequence, $f_2: X\to X_2$ is a $\P^2$-bundle over $\P^2$. Applying \cite{Sato85}, we conclude that $X$ is isomorphic to $\P^2\times \P^2$.
\end{proof}

\subsection{Rank two Fano bundles over $\P^3$} Let us first recall the definition of the null-correlation bundle:

\begin{definition}[{see for instance \cite[Section~4.2]{OSS}, \cite[Example~8.4.1]{Hart78} and \cite{Wev81}}]\label{def:null} Let $\cE$ be a rank $2$ vector bundle on $\P^3$. We say that $\cE$ is a {\it null-correlation bundle} if it fits into an exact sequence
$$
0\rightarrow \cO_{\P^3} \xrightarrow{s} \Omega_{\P^3}(2)\rightarrow \cE(1)\rightarrow 0,
$$ 
where $s$ is a nowhere vanishing section of $\Omega_{\P^3}(2)$.
\end{definition}

In this subsection, we prove the following:

\if0
Let us recall the basic definitions and fundamental facts related to vector bundles on $\P^n$.

Let $\cE$ be a rank $2$ vector bundle on $\P^n$. The Chow ring of $\P^n$ is isomorphic to $\Z[H]/(H^{n+1})$, where $H$ is the class of a hyperplane; then the Chern classes $c_i(\cE)$'s can be considered as integers $c_i$'s, that is, $c_i(\cE)=c_iH^i$. We say that $\cE$ is {\it normalized} if $c_1=0$ or $-1$. 
We also say that $\cE$ is {\it stable} (resp. {\it semistable}) if for every invertible subsheaf $\cL$ of $\cE$, $c_1(\cL)<\dfrac{1}{2} c_1(\cE)$ (resp. $c_1(\cL)\leq \dfrac{1}{2} c_1(\cE)$). 
\fi

\begin{proposition}\label{prop:P1-bundle/P3} Let $X$ be a smooth Fano $4$-fold with nef tangent bundle. Assume that $f_1: X \to \P^3$ is a $\P^1$-bundle. Then $X$ is isomorphic to one of the following:
\begin{enumerate}
\item $\P^1\times \P^3$; 
\item $\P(\cN)$, where $\cN$ is a null-correlation bundle.
\end{enumerate}
\end{proposition}

\begin{proof} Let $X$ be a smooth Fano $4$-fold with nef tangent bundle. Assume that $f_1: X \to \P^3$ is a $\P^1$-bundle. By Corollary~\ref{cor:Brauer}, $f_1: X \to \P^3$ is given by the projectivization of a rank $2$ vector bundle $\cE$ on $\P^3$, that is, $f_1: X=\P(\cE)\to \P^3$. We assume that $\cE$ is normalized and consider its Chern classes $c_1(\cE)$ and $c_2(\cE)$ as integers $c_1$ and $c_2$ respectively. We denote by $H_1$ the ample generator of ${\rm Pic}(\P^3)$, by $F_1$ a fiber of $f_1$ and  by $\xi$ the tautological divisor of $\P(\cE)$.  By the same argument as in \cite[Theorem~2.1]{SW2} and \cite[Example~8.4.1]{Hart78} (see also \cite{Wev81}), we see that one of the following holds:
\begin{enumerate}
\item $\cE$ is isomorphic to $\cO_{\P^3}\oplus \cO_{\P^3}$;  
\item $\cE$ is isomorphic to the null-correlation bundle $\cN$; 
\item $\cE$ is isomorphic to $\cO_{\P^3}\oplus \cO_{\P^3}(-1)$; 
\item $\cE$ is isomorphic to $\cO_{\P^3}(-1)\oplus \cO_{\P^3}(1)$;
\item $\cE$ is isomorphic to $\cO_{\P^3}(-2)\oplus \cO_{\P^3}(1)$;
\item $\cE$ is a stable bundle with $(c_1, c_2)=(0, 3)$; 
\item $\cE$ is a stable bundle with $(c_1, c_2)=(-1, 4)$.
\end{enumerate}

If $\cE$ is isomorphic to $\cO_{\P^3}\oplus \cO_{\P^3}(-1)$, $\cO_{\P^3}(-1)\oplus \cO_{\P^3}(1)$ or $\cO_{\P^3}(-2)\oplus \cO_{\P^3}(1)$, then $X=\P(\cE)$ admits a birational contraction, which contradicts to our assumption that the tangent bundle of $X$ is nef. To prove our assertion, it is enough to show that the cases (vi) and (vii) do not occur. In characteristic $0$, it was proved in \cite[Theorem~2.1]{SW2}, but we do not know whether their argument also holds in positive characteristic or not. 

From now on, we prove that the cases (vi) and (vii) do not occur. Since $X$ is a smooth Fano variety with nef tangent bundle, Theorem~\ref{them:KW21:propert} tells us that there exists another smooth elementary contraction $f_2: X\to X_2$. For any ample divisor $H_2\in {\rm Pic}(X_2)$, there exist integers $a, b$ such that 
$$\cO_X(f_2^{\ast}H_2)\cong \cO_X(a\xi+bf_1^{\ast}H_1)\in {\rm Pic}(X).$$ 

Let us first assume that $\cE$ is a rank $2$ stable bundle on $\P^3$ with $(c_1, c_2)=(0, 3)$. We claim that $a$ and $b$ are not equal to $0$. Taking the intersection number of $a\xi+bf_1^{\ast}H_1$ and $F_1$, we have 
$$
a=(a\xi+bf_1^{\ast}H_1)\cdot F_1=f_2^{\ast}H_2\cdot F_1=H_2\cdot {f_2}_{\ast}(F_1)>0.
$$
Moreover, if $b=0$, then we have $\cO_X(f_2^{\ast}H_2)\cong \cO_X(a\xi)\in {\rm Pic}(X)$; in particular $\xi$ is nef; thus $\cE$ is nef and $c_1=0$. %This implies that $g^{\ast}\cE\cong \cO_{\P^1}^{\oplus 2}$ for any non-constant morphism $g: \P^1\to \P^3$. 
By Corollary~\ref{cor:SRC} (iii), $\cE$ is isomorphic to $\cO_{\P^3}^{\oplus 2}$; this is a contradiction. As a consequence, we see that $b\neq 0$. 

Since we have $\xi^2+c_2f_1^{\ast}H_1^2=0$, $f_1^{\ast}H_1^4=0$ and $\xi\cdot f_1^{\ast}H_1^3=1$, we also have 
$$f_1^{\ast}H_1^4=0,\,\, \xi\cdot f_1^{\ast}H_1^3=1,\,\,\xi^2\cdot f_1^{\ast}H_1^2=0,\,\,\xi^3\cdot f_1^{\ast}H_1=-c_2,\,\,\xi^4=0.$$
By using these equations, we have
$$
0=(f_2^{\ast}H_2)^4=(a\xi+bf_1^{\ast}H_1)^4=4ab(-3a^2+b^2).
$$
Since $a, b\neq 0$, we have
$$
-3a^2+b^2=0.
$$
However this is a contradiction.

Secondly assume that $\cE$ is a rank $2$ stable bundle on $\P^3$ with $(c_1, c_2)=(-1, 4)$. 
By the same argument as in the case where $\cE$ is stable and $(c_1, c_2)=(0, 3)$, we have $$a>0,\,\,f_1^{\ast}H_1^4=0,\,\, \xi\cdot f_1^{\ast}H_1^3=1,\,\,\xi^2\cdot f_1^{\ast}H_1^2=-1,\,\,\xi^3\cdot f_1^{\ast}H_1=-3,\,\,\xi^4=7.$$
Thus we have an equation 
$$
0=(a\xi+bf_1^{\ast}H_1)^4=7a^4-12a^3b-6a^2b^2+4ab^3.
$$
Since $a$ is positive, we have
$$
4\left(\dfrac{\,b\,}{a}\right)^3-6\left(\dfrac{\,b\,}{a}\right)^2-12 \left(\dfrac{\,b\,}{a}\right)+7=0.
$$
Then it follows from the rational root theorem that $\dfrac{\,b\,}{a}=\dfrac{\,1\,}{2}$; this implies that $2\xi+f_1^{\ast}H_1$ is nef. Since $X=\P(\cE)$ is a Fano variety, we have 
$$
0\leq \left(2\xi+f_1^{\ast}H_1 \right)\cdot \left(-K_X \right)^3=\left(2\xi+f_1^{\ast}H_1 \right)\cdot \left(2\xi+5f_1^{\ast}H_1 \right)^3=-232.
$$
This is a contradiction. As a consequence, our assertion holds.
\end{proof}

\subsection{Rank two Fano bundles over $Q^3$} In this subsection, we prove the following:

\begin{proposition}\label{prop:P1bundle:Q3} Let $X$ be a smooth Fano $4$-fold with nef tangent bundle. Assume that $f_1: X \to Q^3$ is a $\P^1$-bundle. Then $X$ does not admit another $\P^1$-bundle structure on $Q^3$.
\end{proposition}

\if0
Before starting the proof of the above proposition, we recall the structure of the Chow ring of $Q^3$. We denote a quadric surface and a line on $Q^3 $by $H$ and $L$ respectively. Then we see that $A_1(Q^3)=\Z[H]$ and $A_2(Q^3)=\Z[L]$. We  also see that 
\fi

\begin{proof} 
Let $X$ be a smooth Fano $4$-fold with nef tangent bundle. Assume that $f_1: X \to X_1:=Q^3$ is a $\P^1$-bundle. To prove Proposition~\ref{prop:P1bundle:Q3}, assume the contrary; then we have another $\P^1$-bundle structure $f_2: X\to X_2:=Q^3$ besides $f_1$. By Corollary~\ref{cor:Brauer}, $f_i: X \to X_i$ ($i=1,2$) is given by the projectivization of a rank $2$ vector bundle $\cE_i$ on $Q^3$. Denoting by $\xi_i$ the divisor corresponding to the tautological line bundle $\cO_{\P(\cE_i)}(1)$ and by $H_i$ the ample generator of ${\rm Pic}(X_i)$, we have 
$${\rm Pic}(X)\cong \Z[\xi_i]\oplus\Z[f_i^{\ast}H_i]\,\,\,\,~\mbox{for~}i=1,\,2.$$
Remark that the Chow groups $A^1(Q^3)$ and $A^2(Q^3)$ are isomorphic to $\Z$; then we consider the Chern classes $c_1(\cE_1), c_1(\cE_2), c_2(\cE_1)$ and $c_2(\cE_2)$ as integers $c_1, c_1', c_2$ and $c_2'$ respectively: 
$$c_1(\cE_1)=c_1H_1,\,\, c_1(\cE_2)=c_1'H_2,\,\, c_2(\cE_1)=\dfrac{c_2}{2}H_1^2,\,\, c_2(\cE_2)=\dfrac{c_2'}{2}H_2^2.$$ 
We may assume that $\cE_i$'s are normalized, that is, $c_1, c_1'\in \{-1, 0\}$. 
We denote by $F_i$ a fiber of $f_i$. 

By Proposition~\ref{prop:bir:sec:contr}, $f_2$ is nothing but the elementary contraction which contracts minimal birational sections of $f_1$ over lines on $Q^3$; this implies that  $$f_1^{\ast}H_1\cdot F_2=1.$$ Then we obtain
$$
2=-K_X\cdot F_2=2\xi_1\cdot F_2+(3-c_1).
$$ 
Comparing the parity of both sides of this equation, we see that $c_1=-1$; thus we obtain $$\xi_1\cdot F_2=-1. $$
Since we have $\xi_1^2+\xi_1\cdot f_1^{\ast}H_1+\dfrac{1}{2}c_2f_1^{\ast}H_1^2=0$, $f_1^{\ast}H_1^4=0$ and $\xi_1\cdot f_1^{\ast}H_1^3=2$, we also have 
$$f_1^{\ast}H_1^4=0,\,\, \xi_1\cdot f_1^{\ast}H_1^3=2,\,\,\xi_1^2\cdot f_1^{\ast}H_1^2=-2,\,\,\xi_1^3\cdot f_1^{\ast}H_1=2-c_2,\,\,\xi_1^4=2c_2-2.$$

For the ample generator $H_2$ of $ {\rm Pic}(X_2)$, there exist integers $a, b$ such that $\cO_X(f_2^{\ast}H_2)\cong \cO_X(a\xi_1+bf_1^{\ast}H_1)\in {\rm Pic}(X)$. By this definition, $(a, b)\neq (0,0)$. Moreover we have 
$$
0=(a\xi_1+bf_1^{\ast}H_1)\cdot F_2=-a+b.
$$
Hence, we obtain $a=b\neq 0$; then we obtain 
$$
0=(\xi_1+f_1^{\ast}H_1)^4=-2c_2+2.
$$
Hence we obtain $c_2=1$. Let us take integers $\alpha, \beta, \gamma, \delta$ as follows:
  \begin{equation}
  \begin{cases}
       \xi_1 &= \alpha \xi_2  + \beta f_2^{\ast}H_2 \\
        f_1^{\ast}H_1&= \gamma \xi_2 +  \delta f_2^{\ast}H_2 
   \end{cases}
  \end{equation}
 Remark that $|\alpha\delta-\beta\gamma|=1$, because $\{\xi_i, f_i^{\ast}H_i\}$ is a $\Z$-basis of ${\rm Pic}(X)$ for $i=1, 2$. 
Since we have 
$$
-1=\xi_1\cdot F_2=\alpha,\quad 1=\xi_1\cdot F_1=-\xi_2\cdot F_1+\beta f_2^{\ast}H_2\cdot F_1
$$
and 
$$
1=f_1^{\ast}H_1\cdot F_2=\gamma,\quad 0  =f_1^{\ast}H_1\cdot F_1   =\xi_2\cdot F_1+\delta  f_2^{\ast}H_2\cdot F_1,
$$
we obtain
  \begin{equation}\label{eq:xiH}
  \begin{cases}
       \xi_1 &= - \xi_2  + \dfrac{1+\xi_2\cdot F_1}{f_2^{\ast}H_2\cdot F_1}f_2^{\ast}H_2 \\
        f_1^{\ast}H_1&= \xi_2 -  \dfrac{\xi_2\cdot F_1}{f_2^{\ast}H_2\cdot F_1}f_2^{\ast}H_2
   \end{cases}
  \end{equation}
 Then the equality $|\alpha\delta-\beta\gamma|=1$ implies that $f_2^{\ast}H_2\cdot F_1=1$. Computing $-K_X\cdot F_1$,  we see that $c_1'=-1$ and $\xi_2\cdot F_1=-1$. Thus the equation (\ref{eq:xiH}) can be written as follows:
   \begin{equation}\label{eq:xiH2}
  \begin{cases}
       \xi_1 &= - \xi_2   \\
        f_1^{\ast}H_1&= \xi_2 +f_2^{\ast}H_2
\end{cases}
\end{equation}

Now we have
\begin{eqnarray}
2=H_2^3&=&f_1^{\ast}H_1\cdot f_2^{\ast}H_2^3\nonumber \\
&=&   f_1^{\ast}H_1\cdot \left(f_1^{\ast}H_1+\xi_1\right)^3 \nonumber \\
&=&f_1^{\ast}H_1^4+3f_1^{\ast}H_1^3\xi_1+3f_1^{\ast}H_1^2\xi_1^2+f_1^{\ast}H_1\xi_1^3=1  \nonumber 
\end{eqnarray}
This is a contradiction.  
\end{proof}

\section{Proof of the main theorem} We prove Theorem~\ref{them:main}. Let $X$ be a smooth Fano $4$-fold defined over an algebraically closed field $k$. Assume that the tangent bundle $T_X$ is nef and the Picard number of $X$ is greater than one. By Theorem~\ref{them:KW21:propert}, we may find a smooth elementary contraction $f: X\to Y$. We denote any fiber of $f$ by $F$; then by Theorem~\ref{them:KW21:propert} again, $F$ and $Y$ are smooth Fano varieties with nef tangent bundle, and we also have $\rho_X=\rho_F+\rho_Y$, $\dim F+\dim Y=4$ and $\dim F, \dim Y >0$. If $X$ admits a contraction onto an FT-manifold $W$, then it follows from Corollary~\ref{cor:FT} that $X$ is isomorphic to $Z\times W$ for some variety $Z$; then $Z$ is one of varieties in Theorem~\ref{them:3-fold}. Thus our assertion holds in this case. By Theorem~\ref{them:3-fold}, we may assume that $Y$ is isomorphic to $\P^2, \P^3$ or $Q^3$. Then by Corollary~\ref{cor:FT}, Proposition~\ref{prop:P2-bundle/P2}, Proposition~\ref{prop:P1-bundle/P3} and Proposition~\ref{prop:P1bundle:Q3}, we may conclude our assertion.

\bibliographystyle{plain}
\bibliography{biblio}
\end{document}